\documentclass[12pt,fleqn]{amsart}

\usepackage{amssymb}
\usepackage{amsmath}
\usepackage{algorithmicx}
\usepackage{bbold}
\usepackage{color}
\usepackage{enumerate}
\newtheorem{theorem}{Theorem}[section]
\newtheorem{lemma}[theorem]{Lemma}
\newtheorem{proposition}[theorem]{Proposition}
\newtheorem{problem}[theorem]{Problem}
\newtheorem{corollary}[theorem]{Corollary}

\theoremstyle{definition}
\newtheorem{definition}[theorem]{Definition}

\theoremstyle{remark}
\newtheorem{remark}[theorem]{Remark}
\newtheorem{claim}[theorem]{Claim}
\numberwithin{equation}{section}

\begin{document}

\title{On non-separable growths of $\omega$ supporting measures}
%\title{On embeddings of Boolean algebras supporting measures to $P(\omega)/Fin$}

%\title[Tukey reductions between partially ordered sets]{Tukey reductions between partially ordered sets}

%    author two information
\author[Piotr Borodulin-Nadzieja]{Piotr Borodulin-Nadzieja}
\address{Instytut Matematyczny, Uniwersytet Wroc\l awski}
\email{pborod@math.uni.wroc.pl}

\author[Tomasz \.{Z}uchowski]{Tomasz \.{Z}uchowski}
\address{Instytut Matematyczny, Uniwersytet Wroc\l awski}
\email{tomasz.zuchowski@math.uni.wroc.pl}

\thanks{The first author was partially supported by National Science Center grant no. 2013/11/B/ST1/03596 (2014-2017).}


\begin{abstract}
	We present several $\mathsf{ZFC}$ examples of compactifications $\gamma\omega$ of $\omega$ such that their remainders $\gamma\omega\backslash\omega$ are nonseparable and carry strictly positive measures. 
\end{abstract}

\maketitle

%%%%%%%%%%%%%%%%%%%%%%%%%%%%%%%%%%%%%%%%%LOCAL SHORTENINGS
%%%%%%%%%%%%%%NUMBERS
\newcommand{\con}{\mathfrak c}
\newcommand{\eps}{\varepsilon}
%%%%%%%%%%%%%%%%%%%%FRAK FAMILIES
\newcommand{\alg}{\mathfrak A}
\newcommand{\algb}{\mathfrak B}
\newcommand{\algc}{\mathfrak C}
\newcommand{\ma}{\mathfrak M}
\newcommand{\pa}{\mathfrak P}
%%%%%%%%%%%%%%%%%%%%SCRIPT FAMILIES
\newcommand{\BB}{\protect{\mathcal B}}
\newcommand{\AAA}{\mathcal A}
\newcommand{\CC}{{\mathcal C}}
\newcommand{\FF}{{\mathcal F}}
\newcommand{\GG}{{\mathcal G}}
\newcommand{\KK}{{\mathcal K}}
\newcommand{\LL}{{\mathcal L}}
\newcommand{\PP}{{\mathcal P}}
\newcommand{\UU}{{\mathcal U}}
\newcommand{\VV}{{\mathcal V}}
\newcommand{\HH}{{\mathcal H}}
\newcommand{\DD}{{\mathcal D}}
\newcommand{\RR}{\protect{\mathcal R}}
\newcommand{\UUU}{{\mathcal U}}
\newcommand{\ide}{\mathcal N}
%%%%%%%%%%%%%%%%%%%%%%%SYMBOLS
\newcommand{\btu}{\bigtriangleup}
\newcommand{\hra}{\hookrightarrow}
\newcommand{\ve}{\vee}
\newcommand{\we}{\cdot}
\newcommand{\de}{\protect{\rm{\; d}}}
\newcommand{\er}{\mathbb R}
\newcommand{\qu}{\mathbb Q}
\newcommand{\supp}{{\rm supp} }
\newcommand{\card}{{\rm card} }
\newcommand{\wn}{{\rm int} }
\newcommand{\ult}{{\rm ULT}}
\newcommand{\vf}{\varphi}
\newcommand{\osc}{{\rm osc}}
\newcommand{\ol}{\overline}
\newcommand{\me}{\protect{\bf v}}
\newcommand{\ex}{\protect{\bf x}}
\newcommand{\stevo}{Todor\v{c}evi\'c}
\newcommand{\cc}{\protect{\mathfrak C}}
\newcommand{\scc}{\protect{\mathfrak C^*}}
\newcommand{\lra}{\longrightarrow}
\newcommand{\sm}{\setminus}
\newcommand{\uhr}{\upharpoonright}
\newcommand{\fin}{\it fin}
\newcommand{\sub}{\subseteq}
\newcommand{\ms}{$(M^*)$}
\newcommand{\m}{$(M)$}
\newcommand{\MA}{MA$(\omega_1)$}
\newcommand{\clop}{\protect{\rm Clop} }

%%MOJE SYMBOLE
\newcommand{\cal}{\mathbb Z}
\newcommand{\zeron}{\cap\{0,\ldots,n\}}
\newcommand{\cantor}{2^{\omega}}
\newcommand{\limN}{\xrightarrow{\scriptscriptstyle n\to\infty}}
\newcommand{\limK}{\xrightarrow{\scriptscriptstyle k\to\infty}}
\newcommand{\limI}{\xrightarrow{\scriptscriptstyle i\to\infty}}
\newcommand{\en}{\mathbb N}
\newcommand{\zet}{\mathcal{Z}_{0}}
\newcommand{\el}{\ell_{1}}
\newcommand{\Tuk}{\preceq_{T}}
\newcommand{\nTuk}{\npreceq_{T}}
\newcommand{\epsi}{\mathcal{E}_{\lambda}}
\newcommand{\cw}{[\mathfrak c]^{<\omega}}
\newcommand{\Bairefin}{\en^{<\omega}}
\newcommand{\kbf}{\mathcal K (\Bairefin) }
\newcommand{\kcz}{\mathcal K (\Czero)}
\newcommand{\cantorfin}{2^{<\omega}}
\newcommand{\Czero}{\mathcal C_{0}}
\newcommand{\Baire}{\omega^{\omega}}
\newcommand{\aort}{\mathcal A^{\perp}}
\newcommand{\KQ}{\mathcal K(\mathbb Q)}
\newcommand{\Tukeq}{\equiv_{T}}
\newcommand{\Tukstr}{\prec_{T}}

%%\newcommand{\epsi}{\mathlarger{\mathlarger{\mathlarger{\boldsymbol{\varepsilon}}}}_{\mu}}
%%%%%%%%%%%%%%%%%%%%%%%%%%%%%%%%%%%%%%%%%%%%%%%%%%%%%%%%%%%

\section{Introduction}\label{introduction}

A compact space $K$ is called a growth of $\omega$ if there is a compactification $\gamma\omega$ of a countable discrete space $\omega$ such that $K$ is homeomorphic to the remainder $\gamma\omega\backslash\omega$. It is well-known that every separable compact space is a growth of $\omega$. Moreover, every such space carries a strictly positive (regular probability Borel) measure, i.e. measure that is positive on every nonempty open subset.

Stone spaces of subalgebras of $\PP(\omega)/\fin$ are natural examples of growths of $\omega$. Recall that due to Parovi\v{c}enko theorem every Boolean algebra of size at most $\omega_1$ can be embedded in $\PP(\omega)/\fin$. In particular, under Continuum
Hypothesis the Lebesgue measure algebra $\mathrm{Bor}([0,1])/_{\lambda=0}$ can be embedded in $\PP(\omega)/\fin$. Since the Stone space of measure algebra is nonseparable and supports a measure, under $\mathsf{CH}$ there is a nonseparable growth of
$\omega$ supporting a measure. On the other hand, under Open Coloring Axiom the measure algebra cannot be embedded in $\PP(\omega)/\fin$ (see \cite{Dow-Hart}). 
Therefore, it is natural to ask the following question:

\begin{problem}[\cite{Drygier-Plebanek15}]
Is there a $\mathsf{ZFC}$ example of a non-separable growth of $\omega$ supporting a measure?
\end{problem}

In terms of Boolean algebras we ask if there is (in $\mathsf{ZFC})$ a non-$\sigma$-centered Boolean algebra
supporting a measure (loosely speaking, a \emph{non-trivial piece} of measure algebra) which can be embedded in $\PP(\omega)/\fin$.

%Under the continuum hypothesis the Stone space $S$ of the Lebesgue measure algebra (the quotient of $Bor([0,1]$) modulo the ideal of Lebesgue null sets) is an example of a nonseparable growth of $\omega$ supporting a measure. However, Dow and Hart proved that the space $S$ is not a growth of $\omega$ if one assumes the Open Coloring Axiom (OCA). 

Notice that compact spaces supporting measures are ccc. Even constructing a $\mathsf{ZFC}$ example of a ccc nonseparable growth of $\omega$ is not a trivial task - in \cite{vanMill79} van Mill offered a bottle of Jenever for such example\footnote{See
	\cite{vanMill79} for the definition of Jenever. Unfortunately
	the authors are not aware of any offer of this kind for finding a $\mathsf{ZFC}$ example of such space which additionally supports a measure.}. The award apparently went to Bell (see \cite[Example 2.1]{Bell80} and \cite[Example 3.2]{vanMill82}) and later on another examples were found (see
	e.g. \cite[Theorem 8.4]{Todorcevic}). Recently, Drygier and Plebanek (\cite{Drygier-Plebanek15}) have constructed a nonseparable growth of $\omega$
carrying a strictly positive measure under the assumption $\mathfrak{b}=\con$.

We present three $\mathsf{ZFC}$ constructions of compactifications of $\omega$ providing a positive answer to the raised question:

\begin{theorem}\label{general}
There exists a compactification $\gamma\omega$ of $\omega$ such that its remainder $\gamma\omega\backslash\omega$ is not separable and supports a measure.
\end{theorem}

Although all of the constructions are the Stone spaces of some Boolean subalgebras of $\PP(\omega)/\fin$, they are of quite different nature. 

%For such a Boolean algebra $\algb$, to prove that $\ult(\algb)$ satisfies the condition required in Theorem \ref{general} it is sufficient to show that $\algb$ is not $\sigma$-centered and there exists a strictly positive finitely additive measure on $\algb$, i.e. measure that is positive on every nonzero element of $\algb$.

The first example (see Section \ref{section:2}) is a subalgebra of $\mathrm{Bor}(\cantor)$ generated by all clopen subsets of $\cantor$ and some sequence $\{U_{\alpha}\colon \alpha<\con\}$ of open sets, where each $U_\alpha$ serves to \emph{kill} certain candidates for a countable dense set in the
resulting Stone space. It is proved that this algebra has required properties with the Lebesgue measure $\lambda$ being strictly positive on it and, moreover, it can be embedded into $\PP(\omega)/\fin$ in such way that $\lambda$ is transferred to the
asymptotic density. This in a sense strengthen the result of Frankiewicz and Gutek that under CH one can construct an embedding of the Lebesgue measure algebra into $\PP(\omega)/\fin$ with such property (see \cite{Frankiewicz-Gutek-81}).

The second example (see Section \ref{section:3}) is constructed using methods from \cite[Theorem 8.4]{Todorcevic}. The non-$\sigma$-centerdness is a result of certain maximality property of the algebra. Contrary to the case of the first example, although it supports a strictly positive measure, it is not clear how
this measure looks like. Its existence is deduced from the Kelley criterion based on the notion of an intersection number. This example appears also in \cite{Pbn-Tanmay} although there the existence of measure is proved using forcing methods. It is
worth to mention that this example can be easily modified to obtain some additional properties, e.g. in \cite{Pbn-Tanmay} it is proved that under $\mathrm{add}(\mathcal{N}) = \mathrm{non}(\mathcal{M})$
%(or, $\mathrm{add}(\mathcal{N})=\mathrm{cov}^*(\mathcal{I}_{1/n})$) 
we can additionally require that this Boolean
algebra does not contain an uncountable independent family and its Stone space has a countable $\pi$-character.

Ironically, after carrying out the above constructions, we have realized that Bell's example of ccc non-separable growth of $\omega$ mentioned above also supports a measure. Bell constructed a non-$\sigma$-centered Boolean subalgebra of $\mathcal{P}(\omega)/\fin$. We show (see Section \ref{section:4}) that it can embedded to $\mathrm{Bor}([0,1])/_{\lambda=0}$. It is, however, not clear if it can be done in a way that the asymptotic density is transferred to the Lebesgue measure.

\section{Preliminaries}\label{preliminaries}
In the sequel, we shall consider Boolean subalgebras of the quotient algebra $\PP(\omega)/\fin$. For every such algebra $\algb$, its Stone space $\ult(\algb)$ (of ultrafilters on $\algb$) is a continuous image of $\ult(\PP(\omega)/\fin)\simeq\beta\omega\backslash\omega$, where by $\beta\omega$ we denote the \v{C}ech-Stone compactification of $\omega$. Therefore, in that case $\ult(\algb)$ is a remainder $\gamma\omega\backslash\omega$ of some compactification $\gamma\omega$ of $\omega$. Recall that $\ult(\algb)$ is separable iff $\algb$ is $\sigma$-centered.

By $\algb^+$ we understand the family of all positive elements of $\algb$. If $\algb$ is a Boolean algebra and $A\in \algb$, then by $\widehat{A}=\{x\in\ult(\algb)\colon A\in x\}$ we denote the corresponding clopen subset of $\ult(\algb)$.

By a measure $\mu$ on a Boolean algebra $\algb$ we mean a finitely additive probability measure. We call it a strictly positive measure if $\mu(B)>0$ for each $B\in\algb^{+}$ (in this case we say also that $\algb$ supports the measure $\mu$). If $\mu$ is a 
measure on $\algb$, then $\widehat{\mu}$  defined by $\widehat{\mu}(\widehat{A})=\mu(A)$ is a finitely additive measure on the algebra of clopen subsets of $\ult(\algb)$. The measure $\widehat{\mu}$ can be uniquely extended to a (regular Borel
probability) $\sigma$-additive measure $\overline{\mu}$ on $\ult(\algb)$. Note that if $\mu$ is strictly positive, then $\overline{\mu}$ is a
strictly positive measure on $\ult(\algb)$.
%($\ult(\algb)$ supports $\overline{\mu}$), i.e. 
%$\overline{\mu}(U)>0$ for any nonempty open subset $U$ of $\ult(\algb)$.

Therefore, Theorem \ref{general} is an immediate consequence of the following statement.

\begin{theorem}\label{general-bool}
There exists a non-$\sigma$-centered Boolean algebra $\algb\sub\PP(\omega)/\fin$ that supports a measure.
\end{theorem}

Recall some standard notions concerning the Cantor space $\cantor$. For any $\varphi\in\cantorfin$ by $[\varphi]=\{x\in\cantor\colon \varphi\subseteq x\}$  we denote the basic subset of $\cantor$ corresponding to $\varphi$. We shall consider the Lebesgue measure $\lambda$ on $\cantor$, defined as the unique Borel probability measure on $\cantor$ such that $\lambda([\varphi])=2^{-k}$ for any $k\in\omega$ and $\varphi\in 2^k$. We say that a subset $A\sub\cantor$ depends on a set of coordinates $I\sub\omega$, which we denote $A\sim I$, if $A=B\times 2^{\omega\backslash I}$ for some $B\sub 2^I$.

We will consider also the asymptotic density defined on elements of the algebra $\PP(\omega)/\fin$. It is usually defined on subsets of $\omega$ as a function 
\[d(A)=\lim_{n\to\infty}\frac{|A\cap n|}{n}\] 
provided the limit exists. As $d(A)=0$ for any finite $A\sub\omega$, we can naturally transfer this function to the asymptotic density on $\PP(\omega)/\fin$, which we will also denote by $d$.

\section{Construction using an almost disjoint family}\label{section:2}

Fix an enumeration $\{P_{\alpha}\colon \alpha<\con\}$ of all countable subsets of $2^\omega$ and let $\{B_{\alpha}\colon \alpha<\con\}$ be an almost disjoint family in $\PP(\omega)$. Denote $\alg_{0}=\clop(\cantor)$.

For $\alpha<\con$ let $P_{\alpha}=\{t_{n}^{\alpha}\colon n\in\omega\}\sub\cantor$ and $B_{\alpha}=\{m_{i}^{\alpha}\colon i\in\omega\}\sub\omega$, where $m_{i}^{\alpha}<m_{j}^{\alpha}$ for $i<j$.  Define a sequence $(\varphi_{n}^{\alpha})_n$ of elements of $\cantorfin$ in the following way: 
\[\varphi_{0}^{\alpha}=t_{0}^{\alpha}|_{\{m_{0}^{\alpha}\}},\text{ } \varphi_{1}^{\alpha}=t_{1}^{\alpha}|_{\{m_{1}^{\alpha}, m_{2}^{\alpha}\}},\text{ } 
\varphi_{2}^{\alpha}=t_{2}^{\alpha}|_{\{m_{3}^{\alpha}, m_{4}^{\alpha}, m_{5}^{\alpha}\}}, \text{ etc.}\]
Define
\[ U_{\alpha}=\bigcup_{i\in\omega}[\varphi_{i}^{\alpha}] \]
and observe that $U_\alpha\subseteq 2^\omega$ is an open set. Let $F_{\alpha}= U_{\alpha}^\mathsf{c}$. 

Finally, let $\alg={\rm alg}\left(\alg_{0}\cup\{U_{\alpha}\colon \alpha<\con\}\right)$.

\begin{theorem}\label{main}
	The Boolean algebra $\alg$ has the following properties
	\begin{itemize}
		\item $\alg$ is not $\sigma$-centered,
		\item $\alg$ supports $\lambda$,
		\item there is a Boolean embedding $\Psi\colon \alg \to \mathcal{P}(\omega)/\fin$ such that $\lambda(A) = d(\Psi(A))$ for each $A\in \alg$.
	\end{itemize}
\end{theorem}

We will prove the above theorem in a sequence of propositions and lemmas. We begin with an easy observation.

\begin{remark}\label{contains}
For each $\alpha<\con$ we have $P_{\alpha}\sub U_{\alpha}$, as $U_{\alpha}$ contains a basic neighbourhood of $t_{n}^{\alpha}$ for every $n\in\omega$.
\end{remark}

\begin{proposition}
The Boolean algebra $\alg$ is not $\sigma$-centered. 
\end{proposition}

\begin{proof}
We will switch to topological language and we will prove that the Stone space of $\alg$ is not separable.
Consider an arbitrary $\{x_{n}\colon n\in\omega\}\sub\ult(\alg)$. For every $n\in\omega$ let $t_{n}=x_{n}|_{\alg_{0}}\in\ult(\alg_{0})\simeq\cantor$, which means that $\{A\in\clop(\cantor)\colon t_{n}\in A\}=\{A\in\clop(\cantor)\colon A\in x_{n}\}$. 
There exists an $\alpha<\con$ such that $\{t_{n}\colon n\in\omega\}=P_{\alpha}$.
By Remark \ref{contains} we have $F_{\alpha}\cap P_{\alpha}=\emptyset$, so $t_{n}\notin F_{\alpha}$ for each $n\in\omega$. As $\cantor$ is a metric $0$-dimensional space, for every $n\in\omega$ there exists $A_{n}\in\clop(\cantor)$ satisfying
$t_{n}\in A_{n}$ and $A_{n}\cap F_{\alpha}=\emptyset$. Therefore, we have $A_{n}\in x_{n}$, hence $F_{\alpha}\notin x_{n}$ for every $n\in\omega$. Thus we have shown that $\widehat{F_{\alpha}}\cap\{x_{n}\colon n\in\omega\}=\emptyset$, which means that
$\{x_{n}\colon n\in\omega\}$ is not dense in $\ult(\alg)$.
\end{proof}

For each $\alpha<\con$ and $i\in\omega$ let $C_{i}^{\alpha}\sub\omega$ be the minimal set with the property $[\varphi_{i}^{\alpha}]\sim C_{i}^{\alpha}$ (so, e.g. $C_0^\alpha = \{m^\alpha_0\}$ and so on). 

\begin{remark}\label{depend}
For every $\alpha<\con$ we have $U_{\alpha}\sim B_{\alpha}$ (thus also $F_{\alpha}\sim B_{\alpha}$). Moreover, $C_{n}^{\alpha}\sub B_{\alpha}\backslash n$ for $\alpha<\con$ and $n\in\omega$.
\end{remark}

To prove that $\alg$ supports the measure $\lambda$, we need the following lemma.

\begin{lemma}\label{positive}
Assume that  $U\sub\cantor$ is an open set and the set $\{\alpha_{1},\ldots,\alpha_{m}\}\sub\con$ is such that 
$U\cap F_{\alpha_{1}}\cap\ldots\cap F_{\alpha_{m}}\neq\emptyset$. Then $\lambda\Big(U\cap F_{\alpha_{1}}\cap\ldots\cap F_{\alpha_{m}}\Big)>0$.
\end{lemma}

\begin{proof}	
Assume that $U$ and $\{\alpha_1, \dots, \alpha_m\}$ are as above. Without loss of generality we can assume that $U=[\tau]$, where $\tau\in 2^M$ for some $M\in \omega$. 

For every $n\in\omega$ define \[ X_{n}=[\tau]\big\backslash\bigcup_{j=1}^{m}\bigcup_{i=0}^{n}[\varphi_{i}^{\alpha_{j}}].\] 
Let \[I_{n}=M\cup\bigcup_{j=1}^{m}\bigcup_{i=0}^{n}C_{i}^{\alpha_{j}} \]
and observe that $I_n$ is the minimal set with property $X_n \sim I_n$.

According to our assumption
$[\tau]\nsubseteq\bigcup_{j=1}^{m}\bigcup_{i\in\omega}[\varphi_{i}^{\alpha_{j}}]$. Hence, for each $n$ the set $X_{n}$ is open and nonempty and so $\lambda(X_{n})>0$.

\begin{claim}\label{disjoint}
There exists $N\in\omega$ such that the family $\{I_{N}\}\cup\{C_{i}^{\alpha_{j}}\colon i>N, j\in\{1,\ldots,m\}\}$ is pairwise disjoint. 
\end{claim}

\begin{proof}
Since the family $\{B_{\alpha}\colon \alpha<\con\}$ is pairwise almost disjoint, there exists $N\geq M$ such that for every different $i, j\in\{1,\ldots,m\}$ we have $B_{\alpha_{i}}\cap B_{\alpha_{j}}\sub N$. Therefore, by Remark \ref{depend} we get 
$C_{k}^{\alpha_{i}}\cap C_{l}^{\alpha_{j}}=\emptyset$ for every different $i,j\in\{1,\ldots,m\}$, $k>N$ and arbitrary $l\in\omega$. If $i\in \{1,\ldots, m\}$ and $k\ne l$, then $C_k^{\alpha_i} \cap C_l^{\alpha_i} = \emptyset$. Moreover, again using Remark \ref{depend} we have $C_{k}^{\alpha_{j}}\sub\omega\backslash M$ for every $k>N$ and $j\in\{1,\ldots,m\}$.
\end{proof} 

Now we prove that the set $ U\cap F_{\alpha_{1}}\cap\ldots\cap F_{\alpha_{m}}= [\tau]\big\backslash\bigcup_{j=1}^{m}\bigcup_{i\in\omega}[\varphi_{i}^{\alpha_{j}}]$ is $\lambda$-positive:

\begin{multline*}
\lambda\Big( [\tau]\big\backslash\bigcup_{j=1}^{m}\bigcup_{i\in\omega}[\varphi_{i}^{\alpha_{j}}]\Big)=\lambda\Big(X_{N}\cap\bigcap_{j=1}^{m}\bigcap_{i=N+1}^{\infty}[\varphi_{i}^{\alpha_{j}}]^\mathsf{c}\Big)\stackrel{\star}{=}
\lambda(X_{N})\prod_{j=1}^{m}\lambda\Big(\bigcap_{i=N+1}^{\infty}[\varphi_{i}^{\alpha_{j}}]^\mathsf{c}\Big)\\\geq\lambda(X_{N})\Big(1-\sum_{i=N+1}^{\infty}\lambda([\varphi_{i}^{\alpha_{j}}])\Big)^m=
\lambda(X_{N})\big(1-2^{-(N+1)}\big)^m>0,
\end{multline*}
where the equality $(\star)$ holds because by Claim \ref{disjoint} the sets $X_N$ and $[\varphi^{\alpha_j}_i]$, for $j\in \{1,\dots,m\}$ and $i>N$,  depend on pairwise disjoint sets of coordinates.
\end{proof}

\begin{proposition}
The measure $\lambda$ is strictly positive on $\alg$.
\end{proposition}

\begin{proof}
	Every element of $\alg$ is a finite union of elements of the form \[ C\cap U_{\beta_{1}}\cap\ldots\cap U_{\beta_{n}}\cap F_{\alpha_{1}}\cap\ldots\cap F_{\alpha_{m}}\] for some $C\in\clop(\cantor)$ and
$\alpha_{1},\ldots,\alpha_{n},\beta_{1},\ldots,\beta_{m}<\con$. Every such element is equal to $U\cap F_{\alpha_{1}}\cap\ldots\cap F_{\alpha_{m}}$ for some open $U\sub\cantor$, hence using Lemma \ref{positive} we get that it is
$\lambda$-positive (provided it is non-empty). Therefore, every nonzero $A\in\alg$ is $\lambda$-positive.
\end{proof}

We are going to find a Boolean embedding of the algebra $\alg$ into $\PP(\omega)/\fin$ transferring the Lebesgue measure to the asymptotic density. 
First, it is easy to see that there is an embedding $\Psi_0\colon \alg_0 \to \PP(\omega)/\fin$ with the above property. For example, it can be induced by $f\colon 2^{<\omega} \to P(\omega)$ such that $f(\sigma) = \{k\in \omega \colon k = \check{\sigma} \mod
2^n\}$, where $n$ is the length of $\sigma$ and $\check{\sigma}$ is the natural number represented by $\sigma$ in the binary system (such $f$ sends $[0]$ to even numbers, $[01]$ to numbers equal $1$ modulo $4$ and so on). 

%As $\cantor$ is separable metric space, there exists  an equidistributed sequence $\{x_{n}:n\in\omega\}$ in $\cantor$ with respect to $\lambda$, i.e. such that for every $f\in C(\cantor)$ we have 
%\[\lim_{n\to\infty}\frac{1}{n}\sum_{j=0}^{n-1}f(x_{j})=\int_{\cantor}f d\lambda.\]
%In particular, for every $C\in\clop(\cantor)$ by taking $f=\chi_{C}$ we get
%\[\lim_{n\to\infty}\frac{|\{m<n:x_{m}\in C\}|}{n}=\lambda(C).\]
%Let us fix such a sequence $(x_{n}:n\in\omega)$. We define $\Psi_{0}:\alg_{0}\to \PP(\omega)/\fin$ by
%\[\Psi_{0}(C)=\{n\in\omega:x_{n}\in C\}^\bullet.\]

%The function $\Psi_{0}$ is obviously a Boolean homomorphism and from the equidistributy we have $\lambda(C)=d(\Psi_{0}(C))$ for every $C\in\clop(\cantor)$, where $d$ is the asymptotic density. Using this we also easily get that $\Psi_{0}$ is $1-1$, as if 
%$\Psi_{0}(C)=\mathbb{0}$ then $d(\Psi_{0}(C))=0$, hence $\lambda(C)=0$ which implies $C=\emptyset$.

We want now to extend $\Psi_{0}$ to a Boolean embedding $\Psi\colon \alg\to\PP(\omega)/\fin$. To define such $\Psi$, we need only to define $\Psi(U_{\alpha})$ for any $\alpha<\con$. We will use the following Lemma \ref{additive} proved as Theorem 1.1
in \cite{Buck53}:

\begin{lemma}\label{additive}
If $A_{0}\sub^{\ast}A_{1}\sub^{\ast}A_{2}\ldots\sub\omega$ is a sequence of sets having an asymptotic density, then there exists a set $A\sub\omega$ having an asymptotic density such that $A_{n}\sub^{\ast}A$ for all $n\in\omega$ and
$d(A)=\sup_{n\in\omega}d(A_{n})$.
\end{lemma}

Fix $\alpha<\con$. Recall that $U_{\alpha}=\bigcup_{i\in\omega}[\varphi_{i}^{\alpha}]$ and let $B_{n}\sub\omega$ be a representative of the equivalence class of $\Psi_{0}\big(\bigcup_{i=0}^{n}[\varphi_{i}^{\alpha}]\big)$ for $n\in\omega$. The sequence $(B_{n}\colon
n\in\omega)$ meets the requirements of Lemma \ref{additive}, thanks to properties of $\Psi_{0}$. Let $B$ be the set, whose existence is proved by Lemma \ref{additive} for the sequence $(B_{n})_n$. We extend $\Psi_{0}$ by defining 
$\Psi_{0}(U_{\alpha})=B^\bullet$.

\begin{proposition}
The function $\Psi_{0}$ defined above on $\alg_{0}\cup\{U_{\alpha}\colon \alpha<\con\}$ can be extended to a Boolean homomorphism $\Psi:\alg\to\PP(\omega)/\fin$.
\end{proposition}

\begin{proof}
%We will use the following lemma by Sikorski (REFERENCE)

%\begin{lemma}\label{Sikorski}
%If $\BB$ is a Boolean algebra generated by the set B and we have a function $f:B\to \CC$, where $\CC$ is another Boolean algebra, then $f$ can be extended to a Boolean homomorphism $\bar{f}:\BB\to\CC$ if the following criterion is satisfied:
%\[\Big(b_{1}^{\epsilon_{1}}\wedge\ldots\wedge b_{n}^{\epsilon_{n}}=\mathbb{0}\Big)\Longrightarrow\Big(f(b_{1})^{\epsilon_{1}}\wedge\ldots\wedge f(b_{n})^{\epsilon_{n}}=\mathbb{0}\Big)\]
%for any $b_{1},\ldots,b_{n}\in B$ and $b_{i}^{\epsilon_{i}}$ meaning $b_{i}$ or $b_{i}^\mathsf{c}$.
%\end{lemma}

By Sikorski's Extension Criterion (see \cite[Theorem 5.5]{Handbook-Boolean}) and the definition of $\alg$ we need only to prove that
\begin{equation}\label{indukcja}
\Big(C\cap\bigcap_{i=1}^{n} U_{\beta_{i}}\cap\bigcap_{j=1}^{m} F_{\alpha_{j}}=\emptyset\Big)\Longrightarrow\Big(\Psi_{0}(C)\cap\bigcap_{i=1}^{n}\Psi_{0}(U_{\beta_{i}})\cap\bigcap_{j=1}^{m}\big(\Psi_{0}(U_{\alpha_{j}})\big)^\mathsf{c}=\mathbb{0}\Big) 
\end{equation}
for every $C\in\clop(\cantor)$ and $\alpha_{1},\ldots,\alpha_{m},\beta_{1},\ldots,\beta_{n}<\con$. We shall proceed by induction on $n$.

For $n=0$ the assumption $C\cap\bigcap_{j=1}^{m} F_{\alpha_{j}}=\emptyset$ means that $C\sub\bigcup_{j=1}^{m} U_{\alpha_{j}}$, so by the definition of $U_{\alpha}$'s we get $C\subseteq\bigcup_{j=1}^{m}\bigcup_{i\in\omega}[\varphi_{i}^{\alpha_{j}}]$.
As $C$ is compact, there exists $N\in\omega$ such that $C\subseteq\bigcup_{j=1}^{m}\bigcup_{i=0}^{N}[\varphi_{i}^{\alpha_{j}}]$, hence by the properties of $\Psi_{0}$ on $\clop(\cantor)$ we have
$\Psi_{0}(C)\leq\bigcup_{j=1}^{m}\Psi_{0}\big(\bigcup_{i=0}^{N}[\varphi_{i}^{\alpha_{j}}]\big)$.
 By Lemma \ref{additive} we know that $\Psi_{0}\big(\bigcup_{i=0}^{n}[\varphi_{i}^{\alpha}]\big)\leq\Psi_{0}(U_{\alpha})$ for each
$\alpha<\con$ and $n\in\omega$. Therefore, we get $\Psi_{0}(C)\leq\bigcup_{j=1}^{m}\Psi_{0}(U_{\alpha_{j}})$ and so $\Psi_{0}(C)\cap\bigcap_{j=1}^{m}\big(\Psi_{0}(U_{\alpha_{j}})\big)^\mathsf{c}=\mathbb{0}$.

Assume now that we have proved (\ref{indukcja}) for every element of the form $C\cap\bigcap_{i=1}^{n} U_{\beta_{i}}\cap\bigcap_{j=1}^{m} F_{\alpha_{j}}$, where $n\leq N$. Take 
 $C\in\clop(\cantor)$ and $\alpha_{1},\ldots,\alpha_{m},\beta_{1},\ldots,\beta_{N+1}<\con$ such that

\begin{equation}\label{ind}
C\cap\bigcap_{i=1}^{N+1} U_{\beta_{i}}\cap\bigcap_{j=1}^{m} F_{\alpha_{j}}=\emptyset.
\end{equation}

\begin{claim}\label{dependk}
There exists such $K\in\omega$ that the sets $[\varphi_{K}^{\beta_{N+1}}]$ and $C\cap\bigcap_{i=1}^{N} U_{\beta_{i}}\cap\bigcap_{j=1}^{m} F_{\alpha_{j}}$  depend on disjoint sets of coordinates.
\end{claim}

\begin{proof}
We have $C\sim l$ for some $l\in\omega$. As $\{B_{\alpha}\colon \alpha<\con\}$ is an almost disjoint family, there exists $K\geq l$ such that $B_{\beta_{N+1}}\backslash K$ is disjoint from $\bigcup_{i=1}^{N}B_{\beta_{i}}\cup\bigcup_{j=1}^{m}B_{\alpha_{j}}$. By 
Remark \ref{depend} we have $[\varphi_{K}^{\beta_{N+1}}]\sim C_{K}^{\beta_{N+1}}\sub B_{\beta_{N+1}}\backslash K$ and \[ C\cap\bigcap_{i=1}^{N} U_{\beta_{i}}\cap\bigcap_{j=1}^{m} F_{\alpha_{j}}\sim
l\cup\bigcup_{i=1}^{N}B_{\beta_{i}}\cup\bigcup_{j=1}^{m}B_{\alpha_{j}},\] thus the proof is complete.
\end{proof}

Fix such $K\in\omega$. As $[\varphi_{K}^{\beta_{N+1}}]\sub U_{\beta_{N+1}}$, by (\ref{ind}) we get 
\[\Big(C\cap\bigcap_{i=1}^{N} U_{\beta_{i}}\cap\bigcap_{j=1}^{m} F_{\alpha_{j}} \Big)\cap[\varphi_{K}^{\beta_{N+1}}]=\emptyset.\]
If two sets depending on disjoint sets of coordinates have empty intersection, then at least one of them has to be empty. 
%By Claim \ref{dependk} we have two sets with empty intersection which depend on disjoint set of coordinates, which means that one of them must be empty. 
As $[\varphi_{K}^{\beta_{N+1}}]\neq\emptyset$, we get
\[C\cap\bigcap_{i=1}^{N} U_{\beta_{i}}\cap\bigcap_{j=1}^{m} F_{\alpha_{j}}=\emptyset.\]
Now, by inductive assumption 
\[\Psi_{0}(C)\cap\bigcap_{i=1}^{N}\Psi_{0}(U_{\beta_{i}})\cap\bigcap_{j=1}^{m}\big(\Psi_{0}(U_{\alpha_{j}})\big)^\mathsf{c}=\mathbb{0},\]
hence
\[\Psi_{0}(C)\cap\bigcap_{i=1}^{N+1}\Psi_{0}(U_{\beta_{i}})\cap\bigcap_{j=1}^{m}\big(\Psi_{0}(U_{\alpha_{j}})\big)^\mathsf{c}=\mathbb{0}.\]
Thus, we have proved (\ref{indukcja}) and we are done.
%hence by Sikorski's extension criterion we can extend our homomorphism $\Psi_{0}$.
\end{proof}

Now we shall prove that the homomorphism $\Psi\colon \alg\to\PP(\omega)/\fin$ transfers the Lebesgue measure to the asymptotic density. In this proof we will not use any specific properties of $U_{\alpha}$'s, but rather the properties of $\Psi_{0}$ on $\alg_{0}$ and the definition of $\Psi(U_{\alpha})$ for $\alpha<\con$.

\begin{proposition}\label{transfer}
For every $A\in\alg$ we have $\lambda(A)=d(\Psi(A))$.
\end{proposition}

\begin{proof}
Let \[\alg' = \{A\in \alg\colon \lambda(A \cap C) = d(\Psi(A \cap C)) \mbox{ for each }C\in \mathrm{Clop}(2^\omega)\}.\] 

\begin{claim}
If $A\in \alg'$, then $A\cap U_\alpha \in \alg'$ and $A\cap F_\alpha \in \alg'$ for each $\alpha<\con$.
\end{claim}

\begin{proof}
	Let $A\in \alg'$ and $\alpha<\con$. We will only show that $A\cap U_\alpha \in \alg'$. The proof for $F_\alpha$ is analogous. Fix $C\in \mathrm{Clop}(2^\omega)$.

Denote $C_{n}=\bigcup_{i=0}^{n}[\varphi_{i}^{\alpha}]$, so that we have $U_{\alpha}=\bigcup_{n=0}^{\infty}C_{n}$, where the union is increasing. 
Since $A\in \alg'$ and $C \cap C_n$ is clopen for each $n$ we have
\begin{equation}\label{jeden}
\lambda\Big(A\cap C \cap U_{\alpha}\Big)=\sup_{n\in\omega}\lambda\Big(A\cap C \cap C_{n}\Big)=\sup_{n\in\omega}d\bigg(\Psi\Big(A\cap C \cap C_{n}\Big)\bigg).
\end{equation}
Denote $A' = A \cap C$. As $\Psi$ is a homomorphism, for every $n\in\omega$ we have 
$\Psi(A' \cap C_{n})\leq\Psi(A' \cap U_{\alpha})$,
and thus
\begin{equation}\label{psidown}
\sup_{n\in\omega}d\bigg(\Psi\Big(A' \cap C_{n}\Big)\bigg)\leq d\bigg(\Psi\Big(A' \cap U_\alpha\Big)\bigg).
\end{equation}
Also, $A' \cap U_{\alpha} = (A' \cap C_{n})\cup \Big(A' \cap(U_{\alpha}\backslash C_{n})\Big)$ for each $n\in\omega$
and hence
\begin{equation}\label{psiup}
\Psi(A' \cap U_{\alpha}) = \Psi(A' \cap C_{n})\cup \Big(\Psi(A')\cap\Psi(U_{\alpha}\backslash C_{n})\Big).
\end{equation}
By the definition of $\Psi$ we get $d(\Psi(U_{\alpha}))=\sup_{n\in\omega}d(\Psi(C_{n}))$ for every $\alpha<\con$ and so
\[d\Big(\Psi(U_\alpha \backslash C_{n})\Big)\limN 0.\]

Therefore, we have 
\begin{equation}\label{oszac}
 d\Big(\Psi(A')\cap\Psi(U_{\alpha}\backslash C_{n})\Big)\limN 0. 
\end{equation}
Finally, by (\ref{psidown}), (\ref{psiup}), (\ref{oszac}) and the definition of $A'$ we get
\[d\bigg(\Psi\Big(C\cap A \cap U_{\alpha}\Big)\bigg)=\sup_{n\in\omega}d\bigg(\Psi\Big(C\cap A \cap C_{n}\Big)\bigg),\]
and so by (\ref{jeden}) we have
\begin{equation}\label{open}
d\bigg(\Psi\Big(C\cap (A \cap U_{\alpha})\Big)\bigg)=\lambda\Big(C\cap(A\cap U_{\alpha})\Big).
\end{equation}
As $C$ was arbitrary, $A\cap U_\alpha\in \alg'$. 

We can show that $A\cap F_\alpha\in \alg'$, using the fact that $\lambda\Big(A\cap C \cap F_{\alpha}\Big)=\inf_{n\in\omega}\lambda\Big(A\cap C \cap D_{n}\Big)$ for every $C\in \mathrm{Clop}(2^\omega)$
and $D_n = C^c_n$ for each $n$ and proceeding in the similar manner as above.
\end{proof}
Now, subsequently using the above claim we can show that if \[ A = C\cap U_{\beta_{1}}\cap\ldots\cap U_{\beta_{k}}\cap F_{\alpha_{1}}\cap\ldots\cap F_{\alpha_{m}},\] where $C\in\clop(\cantor)$ and
$\alpha_{1},\ldots,\alpha_{m},\beta_{1},\ldots,\beta_{k}<\con$, then $A\in \alg'$ and so, in particular, $\lambda(A) = d(\Psi(A))$.
Finally, as every element of $\alg$ can be written as a disjoint union of elements of the form as above, we prove the statement of Proposition \ref{transfer} by the fact that the Lebesgue measure and the asymptotic density are finitely additive.
\end{proof}

\begin{corollary}
The homomorphism $\Psi\colon \alg\to\PP(\omega)/\fin$ is injective and so it is a Boolean embedding.
\end{corollary}
\begin{proof}
Assume that $A\in\alg$ satisfies $\Psi(A)=\mathbb{0}$. We have $d(\Psi(A))=0$, hence by Proposition \ref{transfer} we get $\lambda(A)=0$, which implies, by Proposition \ref{positive}, that $A=\emptyset$.
\end{proof}

\section{Construction using slaloms.}\label{section:3}

In this section we present another construction announced in the introduction. It is motivated by \cite[Theorem 8.4]{Todorcevic}. In \cite{Pbn-Tanmay} the authors showed that the constructed Boolean algebra supports a strictly positive measure using the following theorem due to
Kamburelis: 

\begin{theorem}
	A Boolean algebra $\mathfrak{A}$ supports a measure if and only if there is a measure algebra $\mathbb{M}$ such that $\Vdash_\mathbb{M} "\check{\mathfrak{A}}\mbox{ is }\sigma\mbox{-centered}"$. 
\end{theorem}

We will prove the existence of strictly positive measure using a more classical tool: Kelley's criterion.

\begin{definition}
	Let $\mathfrak{A}$ be a Boolean algebra and let $\mathcal{A}\subseteq \mathfrak{A}^{+}$ be nonempty. For every finite sequence $s=(A_{1},\ldots,A_{n})$ of elements of $\AAA$ let
\[\kappa(s)=\frac{max\{|I|\colon I\sub\{1,\ldots,n\}\text{ and }\bigcap_{i\in I}A_{i}\neq\mathbb{0}\}}{n}.\]

Define the \emph{intersection number} of $\AAA$ by
\[\kappa(\AAA)=\inf\{\kappa(s)\colon s\in\AAA^{<\omega}\}.\]
\end{definition}

\begin{theorem}[Kelley's criterion]
	Let $\mathfrak{A}$ be a Boolean algebra. The following conditions are equivalent:
	\begin{itemize}
		\item $\mathfrak{A}$ supports a measure,
		\item $\mathfrak{A}^+ = \bigcup_n \mathcal{C}_n$, where $\kappa(\mathcal{C}_n)>0$ for each $n$. \label{Kel}
	\end{itemize}
\end{theorem}

We will consider a slightly stronger property than the above (see \cite{Dzamonja-Plebanek-08}).

\begin{definition} We call a Boolean algebra $\mathfrak{A}$ \emph{appoximable} if for every $0<\delta<1$ there is a family $(\mathcal{C}_n)_n$ such that $\mathfrak{A}^+ = \bigcup \mathcal{C}_n$ and $\kappa(\mathcal{C}_n)>\delta$.
\end{definition}

A \emph{slalom} is a set $S\subseteq \omega\times\omega$ such that $S(n)\subseteq 2^n$ and $|S(n)|< 2^n$ for every $n
\in\omega$. Denote by $\mathcal{S}$ the family of all slaloms.
For a slalom $S$ denote $S_{| n} = S\cap (n\times 2^n)$. Let \[\Omega = \{(S,n)\colon n\in\omega, \ S\in \mathcal{S}, S\subseteq (n\times 2^n)\}.\]
For each $S\subseteq \omega\times\omega$ define 
\[ T_S = \{(T,n)\in \Omega \colon S_{| n} \subseteq T\}. \]
For $(S,n)\in \Omega$ let
\[ T_{(S,n)} = \{(T,m)\in \Omega \colon m\geq n, T_{| n} = S\}. \]

It will be convenient to make the following simple observations available.

\begin{lemma}\label{ll}
	Let $A$, $B\in \mathcal{S}$. Then
\begin{enumerate}
	\item $A\subseteq B$ if and only if $T_B \subseteq T_A$,
	\item\label{two} $T_{(A\cup B)} = T_A \cap T_B$,
	\item\label{intersection} if $\mathcal{F}$ is finite and $\bigcap_{F\in\mathcal{F}} T_F$ is finite, then there is $k\in \omega$ such that $\bigcup_{F\in \mathcal{F}} F(k) = 2^k$,
	\item\label{centered} if $\mathcal{A}$ is such that $\{T_A\colon A\in \mathcal{A}\}/\fin$ is centered, then $\bigcup \mathcal{A}\in \mathcal{S}$.
\end{enumerate}
\end{lemma}

Let
\[ \mathcal{W} = \{S\in \mathcal{S}\colon \sum |S(n)|/2^n < \infty \}. \]

Now, we are ready to define the main object of this section. Let 
\[ \mathfrak{T} = {\rm alg}\left(\{T_A \colon A\in \mathcal{W}\} \cup \{T_{(S,n)}\colon (S,n)\in\Omega\}\right). \]
Let $K$ be the Stone space of $\mathfrak{T}/\fin$. Since $\Omega$ is countable, $\mathfrak{T}/\fin$ can be embedded in $\mathcal{P}(\omega)/\fin$ and so $K$ is a growth of $\omega$.

\begin{theorem}
The Boolean algebra $\mathfrak{T}/\fin$ is an approximable non-$\sigma$-centered Boolean algebra. Consequently, $K$ is a growth of $\omega$ supporting a measure.
\end{theorem}

Let \[ \mathcal{X} = \{f\in\omega^\omega\colon \forall n \ f(n) < 2^n\}. \]

\begin{proposition}$\mathfrak{T}/\fin$ is not $\sigma$-centered.
\end{proposition}
\begin{proof}
	Suppose towards contradiction that $\mathfrak{T}/\fin$ is $\sigma$-centered. In particular $\mathcal{W} = \bigcup \mathcal{W}_n$, where $\{T_W\colon W\in \mathcal{W}_n\}/\fin$ is centered for each $n$. By Lemma \ref{ll} (\ref{centered}) $W_n = \bigcup
	\mathcal{W}_n \in \mathcal{S}$ for each $n$. Now, pick $f\in \mathcal{X}$ to be such that $f(n)\notin W_n(n)$. Clearly, $f\in \mathcal{W}$ but $f\notin \mathcal{W}_n$ for every $n$. A contradiction.
 \end{proof}

\begin{proposition}
	$\mathfrak{T}/\fin$ is approximable.
\end{proposition}

\begin{proof}
	Let $0<\delta<1$. For $(S,n)\in \Omega$ define
	\[ \mathcal{W}^\delta_{(S,n)} = \{W\in \mathcal{W}\colon W_{| n}=S \mbox{ and } \sum_{k>n} |W(k)|/2^k < 1 - \delta\}. \]
It is easy to see that 
\[ \mathcal{W} = \bigcup_{(S,n)\in \Omega} \mathcal{W}^\delta_{(S,n)}. \]
We claim that even something stronger is true. 
\bigskip

\begin{claim} \label{cl1}For each infinite $A\in \mathfrak{T}$ there is $(S,n)\in \Omega$ and $V\in \mathcal{W}^\delta_{(S,n)}$ such that $T_V \cap T_{(S,n)} \subseteq A$.
\end{claim}

\begin{proof}It is sufficient to consider only elements of the form
\[ A = \big(T^c_{V_0} \cap T^c_{V_1} \cap \dots \cap T^c_{V_l}\big) \cap \big(T_{V_{l+1}}\cap \dots \cap T_{V_L}\big). \]
Also, thanks to Lemma \ref{ll}(\ref{two}) 
\[ A =  \big(T^c_{V_0} \cap T^c_{V_1} \cap \dots \cap T^c_{V_l}\big) \cap T_V, \]
where $V = V_{l+1} \cup \dots \cup V_L\in \mathcal{W}$. Now, since $A$ is infinite, $V_i\nsubseteq V$ for each $i\leq l$. Let $n$ be big enough so that $V_i \cap (n\times 2^\omega) \nsubseteq V_{|n}$. Let $S = V_{| n}$. It is plain to check that $T_V
\cap T_{(S,n)} \subseteq A$.
\end{proof}

\begin{claim} \label{cl2}$\kappa(\{T_W \cap T_{(S,n)}\colon W\in \mathcal{W}^\delta_{(S,n)}\}/\fin)>\delta$ for each $(S,n)\in \Omega$.\end{claim}

\begin{proof}
Fix $(S,n)\in \Omega$. For $W\in \mathcal{W}^\delta_{(S,n)}$ let \[ A_W = \{f\in \mathcal{X}\colon f(k)\notin W(k) \mbox{ for } k>n\}. \]
Of course $\lambda(A_{(i,j)}) = 1-1/2^i$ for $i>k$. Since \[ A_W = \bigcap_{(i,j)\in W, i>n} A_{(i,j)} \] and $\sum_{i>n} |W(i)|/2^i < 1 -\delta$ we have that \[ \lambda(A_W)>\delta. \]
Let $(V_i)_{i<k}$ be a sequence of elements of $\mathcal{W}^\delta_{(S,n)}$. There is $I\subseteq k$ such that $|I|\geq \delta \cdot k$ and there is $f\in \bigcap_I A_{V_i}$. Therefore, $\bigcup_I V_i \in \mathcal{W}$, just because $f(k)\notin
\bigcup_I V_i(k)$ for $k>n$. By Lemma \ref{ll}(\ref{intersection}) $\bigcap_I T_{V_i}$ is infinite. Moreover, $V_i \cap (n\times 2^n)=S$ for every $i\leq l$ and so  $\bigcap_I T_{V_i} \cap T_{(S,n)}$ is infinite and the claim is proved.
\end{proof}

For $n\in \omega$ let  \[ \mathcal{C}_{(S,n)} = \{A\in \mathfrak{A}\colon \exists W\in\mathcal{W}^\delta_{(S,n)} \ T_W \cap T_{(S,n)} \subseteq^* A\}/\fin. \]
Claim \ref{cl1} implies that $(\mathcal{C}_{(S,n)})_{(S,n)\in\Omega}$ is a fragmentation of $\mathfrak{A}$. By Claim \ref{cl2} $\kappa(\mathcal{C}_{(S,n)})>\delta$ for each $(S,n)\in \Omega$.
\end{proof}

\section{Bell's construction} \label{section:4}

In this section we will describe Bell's construction of a ccc non-separable growth of $\omega$ from \cite{Bell80} (mentioned in the introduction) and we will show that it supports a measure.

Let $P=\{f\in\Baire\colon  f(n)\leq n+1 \text{ for each } n\in\omega\}$ and $N=\{f\restriction n\colon  f\in P,n\in\omega\}$. 
Denote $T=\{\pi\in N^{\omega}\colon \pi(n)\in\omega^{n+1} \text{ for each } n\in\omega\}$.

 For each $s\in N$ define $C_{s}=\{t\in N\colon s\sub t\}$ and for every $\pi\in T$ let 
\[C_{\pi}=\bigcup_{n\in\omega}C_{\pi(n)}.\]

Finally, let $\algb={\rm alg}\left(\{C_{\pi}\colon \pi\in T\}\right)$. Since $N$ is a countably infinite set, $\algb/\fin$ can be embedded to $\mathcal{P}(\omega)/\fin$ and so the Stone space of $\algb/\fin$ is a growth of $\omega$. It is not difficult to see that $\algb/\fin$ is not $\sigma$-centered (see \cite{Bell80}). It is also ccc. In fact Bell proved that this space is $\sigma$-$n$-linked for each $n\in \omega$, i.e. for every
$n$ we have $\algb^+ = \bigcup_i \mathcal{C}_i$ where $\mathcal{C}_i$ is $n$-linked for every $i$ (i.e. $\bigcap \FF\ne \emptyset$ whenever $\FF
\in [\mathcal{C}_i]^n$). Plainly, $\sigma$-$n$-linked spaces are ccc. 

We will show that Bell's space supports a measure. More precisely, we will show that $\algb/\fin$ is isomorphic to a certain subalgebra of $\mathrm{Bor}([0,1])/_{\lambda=0}$. This implies $\sigma$-$n$-linkedness (see also \cite{Dow-Steprans-94}), so our
theorem generalizes Bell's result.

Endow $X=\prod_{n\in\omega}\{0,\ldots,n+1\}$ with the product topology and notice that $X$ is homeomorphic to the Cantor set. For each $s\in N$ let $[s]=\{t\in X\colon s\subseteq t\}$ (the basic open subset of $X$ corresponding to $s$).
For every $\pi\in T$ define $V_{\pi}=\bigcup_{n\in\omega}[\pi(n)]$. Finally, let 
$\algc={\rm alg}\left(\{V_{\pi}\colon \pi\in T\}\right)$.
 
\begin{proposition}
The Boolean algebra $\algb/\fin$ is isomorphic to $\algc$.
\end{proposition}

\begin{proof}
Define $f\colon \{C_{\pi}\colon \pi\in T\}\to\{V_{\pi}\colon \pi\in T\}$ by $f(C_{\pi})=V_{\pi}$ for $\pi\in T$. We claim that such $f$ can be extended to a function $\bar{f}\colon \algb\to\algc$ inducing a Boolean isomorphism of $\algb/\fin$ and $\algc$.

By the Sikorski's Extension Criterion we only need to prove that
\[\bigcap_{i=1}^{n}C_{\pi_{i}^{\prime}}\cap\bigcap_{j=1}^{m}C_{\pi_{j}}^\mathsf{c} \text{ is finite} \iff\bigcap_{i=1}^{n}V_{\pi_{i}^{\prime}}\cap\bigcap_{j=1}^{m}V_{\pi_{j}}^\mathsf{c}=\emptyset\]
for every $\pi_{1},\ldots,\pi_{m},\pi_{1}^\prime,\ldots,\pi_{n}^\prime\in T$.

First, assume that there exists $t\in\bigcap_{i=1}^{n}V_{\pi_{i}^{\prime}}\cap\bigcap_{j=1}^{m}V_{\pi_{j}}^\mathsf{c}$ for some $\pi_{1},\ldots,\pi_{m}$,\\ $\pi_{1}^\prime,\ldots,\pi_{n}^\prime\in T$. 
As $\bigcap_{i=1}^{n}V_{\pi_{i}^{\prime}}$ is an open subset of $X$, there exists $k\in\omega$ such that 
$[t|k]\sub\bigcap_{i=1}^{n}V_{\pi_{i}^{\prime}}$, thus $t|k\in\bigcap_{i=1}^{n}C_{\pi_{i}^{\prime}}$.

Since $t$ does not extend any of $\pi_{j}(i)$ for $j=1,\ldots,m$ and $i\in\omega$, for each $l\geq k$ we have 
\[t|l\in C_{t|k}\cap\bigcap_{j=1}^{m}C_{\pi_{j}}^\mathsf{c}\sub\bigcap_{i=1}^{n}C_{\pi_{i}^{\prime}}\cap\bigcap_{j=1}^{m}C_{\pi_{j}}^\mathsf{c},\] hence the latter set is infinite.

On the other hand, assume that $D = \bigcap_{i=1}^{n}C_{\pi_{i}^{\prime}}\cap\bigcap_{j=1}^{m}C_{\pi_{j}}^\mathsf{c}$ is infinite for some $\pi_{1},\ldots,\pi_{m}$,\\ $\pi_{1}^\prime,\ldots,\pi_{n}^\prime\in T$. 

\begin{claim}\label{branch}
There is $t\in X$ and $M\in \omega$ such that $t|l\in D$ for each $l\geq M$.% \bigcap_{i=1}^{n}C_{\pi_{i}^{\prime}}\cap\bigcap_{j=1}^{m}C_{\pi_{j}}^\mathsf{c}$ for every $i$.
\end{claim}

\begin{proof}
	We will construct $t$ inductively by finding the sequence of its initial segments $(t_k)_{k\geq M}$. First, notice that as $D$ is infinite, it must contain arbitrarily long sequences from $N$. In particular, there is $M>m$ and $t_M \in \omega^M \cap D$.  
Now, assume that we have $t_k$ of length $k$, such that $t_M \subseteq t_k$ and $t_k \in D$. We will show that there is $p\leq k+1$ such that $t_k \text{\textasciicircum} p\in D$ and that will finish the proof. 

Clearly, $t_k \text{\textasciicircum} r \in \bigcap_{i=1}^{n}C_{\pi_{i}^{\prime}}$ for every $r\leq k+1$. So, we need only to find $p\leq k+1$ such that $t_k \text{\textasciicircum} p \in \bigcap_{j=1}^{m}C_{\pi_{j}}^\mathsf{c}$. Notice that if
$t_k \text{\textasciicircum} r$ does not extend $\pi_j(k)$ for some $j\leq m$, then $t_k \text{\textasciicircum} r \in C_{\pi_{j}}^\mathsf{c}$. Hence, as $k>m$, there is $p\leq k+1$ such that $t_k \text{\textasciicircum} p \in \bigcap_{j=1}^{m}C_{\pi_{j}}^\mathsf{c}$.

Let $t$ be the unique element of $X$ such that $t_k\subseteq t$ for each $k\geq M$.
\end{proof}

Let $t$ be as in Claim \ref{branch}. Then $t\in\bigcap_{i=1}^{n}V_{\pi_{i}^{\prime}}$, since $t$ extends $\pi_{i}^\prime(l)$ for each $i=1,\ldots,n$ and $l\in\omega$.
Also, $t\in\bigcap_{j=1}^{m}V_{\pi_{j}}^\mathsf{c}$. Otherwise there would be $j\in\{1,\ldots,m\}$ and $l\in\omega$ such that $t$ extends $\pi_{j}(l)$. But then $t|r$ would extend $\pi_{j}(l)$ for some $r>M$, a contradiction with Claim
$\ref{branch}$. So, \[\bigcap_{i=1}^{n}V_{\pi_{i}^{\prime}}\cap\bigcap_{j=1}^{m}V_{\pi_{j}}^\mathsf{c}\neq\emptyset\] and we are done.

%Reassuming, we get $\bigcap_{i=1}^{n}V_{\pi_{i}^{\prime}}\cap\bigcap_{j=1}^{m}V_{\pi_{j}}^\mathsf{c}\neq\emptyset$ which ends the proof.
\end{proof}

The above lemma allows us to look for a strictly positive measure on $\algc$ instead of $\algb/\fin$.
Denote by $\lambda$ the standard product measure on $X$ (defined by $\lambda([s])=\frac{1}{(n+1)!}$ for $s\in N\cap\omega^{n}$).

\begin{proposition}
The measure $\lambda$ is strictly positive on $\algc$.
\end{proposition}

\begin{proof}
Notice that the set of nonempty elements of 
\[ \mathcal{P} = \{[s]\cap \bigcap_{j=1}^m V^c_{\pi_j}\colon s\in N, \pi_1, \dots \pi_m \in T\} \]
forms a $\pi$-base for $\algc$. 
So it is enough to show that every nonempty element of $\mathcal{P}$ is $\lambda$-positive or, equivalently, that 
\[ \lambda\big([s]\cap \bigcup_{j=1}^m V_{\pi_j} \big) < \lambda([s]) \]
for $s\in N, \pi_1, \dots \pi_m \in T$ such that $[s]\cap \bigcap_{j=1}^m V_{\pi_j}^c \ne \emptyset$.

Suppose towards contradiction that
\begin{equation}\label{wypelnia}
\lambda\Big([s]\cap\bigcup_{j=1}^{m}V_{\pi_{j}}\Big)=\lambda([s])
\end{equation}
for some $s\in\omega^K\cap N$, where $K\in\omega$, and $\pi_1, \dots \pi_m \in T$ satisfying $[s]\cap \bigcap_{j=1}^m V_{\pi_j}^c \ne \emptyset$.

For each $n\geq 1$ denote $A_{n}=\bigcup_{j=1}^{m}\bigcup_{i=0}^{n-1}[\pi_{j}(i)]$ and notice that $A_{n}$ is a disjoint union of sets of the form $[t]$, $t\in\omega^n$. 
Therefore, for each $n>K$ the set $[s]\cap A_{n}^\mathsf{c}$ is also a disjoint union of such basic subsets. By the assumption $\emptyset\neq[s]\cap\bigcap_{j=1}^{m}V_{\pi_{j}}^\mathsf{c}$ and so the set $[s]\cap A_{n}^\mathsf{c}$ is nonempty.
Therefore, 
for every $n>K$ we have
\begin{equation}\label{duzo}
\lambda\big([s]\cap A_{n}^\mathsf{c}\big)\geq\frac{1}{(n+1)!}.
\end{equation}

For every $l\in\omega$ the set $A_{l+1}\backslash A_{l}$ is a disjoint union of at most $m$ subsets of form $[t]$, where $t\in\omega^{l+1}$, thus 
\begin{equation}\label{malo}
\lambda\big(A_{l+1}\backslash A_{l}\big)\leq m\cdot\frac{1}{(l+2)!}.
\end{equation}

\begin{lemma}\label{taylor}
For every $n>3m$ we have
\[m\cdot\sum_{l=n}^{\infty}\frac{1}{(l+1)!}<\frac{1}{n!}.\]
\end{lemma}

\begin{proof}
The series $\sum_{l=n}^{\infty}\frac{1}{(l+1)!}$ is a remainder term at $x=1$ of the $n$-th order Taylor polynomial of the function $g(x)=e^x$ at $0$. The Lagrange form of the remainder gives us
\[\sum_{l=n}^{\infty}\frac{1}{(l+1)!}=\frac{e^x}{(n+1)!}\] 
for some $x\in[0,1]$, thus
\[m\cdot\sum_{l=n}^{\infty}\frac{1}{(l+1)!}\leq\frac{e\cdot m}{n}\cdot\frac{1}{n!},\]
and we are done.
\end{proof}

Fix $n>max(K,3m)$. Using (\ref{duzo}), (\ref{malo}) and Lemma \ref{taylor} we get
\begin{multline*}\lambda\Big([s]\cap\bigcup_{j=1}^{m}V_{\pi_{j}}\cap A_{n}^\mathsf{c}\Big)=\lambda\Big([s]\cap\bigcup_{l=n}^{\infty}(A_{l+1}\backslash A_{l})\Big)\leq m\cdot\sum_{l=n}^{\infty}\frac{1}{(l+2)!}\\<\frac{1}{(n+1)!}\leq\lambda\big([s]\cap
	A_{n}^\mathsf{c}\big),
\end{multline*}

which contradicts the assumption (\ref{wypelnia}). 

\end{proof}

\section{Acknowledgements} 

We would like to thank the participants of the $\mathsf{BF}$-seminar in Wroc\l aw, Grzegorz Plebanek and Piotr Drygier, for helpful discussions concerning the subject of this paper. 

\bibliographystyle{alpha}
\bibliography{smallb}

\end{document}